\documentclass[10pt]{article}
\usepackage{amsmath}
\usepackage{amssymb}
\usepackage{amsthm}
\usepackage{cite}
\usepackage[margin = 1in]{geometry}
\usepackage{graphicx}
\usepackage{hyperref}
\usepackage[utf8]{inputenc}
\usepackage{lipsum}
\usepackage{mathrsfs}
\usepackage{parskip}
\usepackage{setspace}
\usepackage{soul}
\usepackage{tikz}
\usepackage{url}

\newtheorem{theorem}{Theorem}
\newtheorem{lemma}[theorem]{Lemma}
\newtheorem{corollary}[theorem]{Corollary}
\newtheorem*{remark}{Remark}

\title{Enhancing Neural Network Differential Equation Solvers}
\author{Matthew J. H. Wright}
\date{December 2022}

\begin{document}

\maketitle

\begin{abstract}

We motivate the use of neural networks for the construction of numerical solutions to differential equations. We prove that there exists a feed-forward neural network that can arbitrarily minimise an objective function that is zero at the solution of Poisson's equation, allowing us to guarantee that neural network solution estimates can get arbitrarily close to the exact solutions. We also show how these estimates can be appreciably enhanced through various strategies, in particular through the construction of error correction networks, for which we propose a general method. We conclude by providing numerical experiments that attest to the validity of all such strategies for variants of Poisson's equation. The source code for this project can be found at \url{https://github.com/mjhwright/error-correction}.

\end{abstract}

\section{Introduction}

Differential equations are among the most ubiquitous problems in contemporary mathematics. In recent years, developments in artificial neural networks have prompted new research into their capacity to be excellent differential equation solvers \cite{NNDE_Pros, PINNS1, PINNS2, Sirignano_2018, Quanta}. They are universal approximators \cite{Hornik1991}; they can circumvent the curse of dimensionality \cite{Barron}; and they are continuous. However, practically, their construction and optimisation costs are enough to deter the discerning user.

In this paper, we explain a method by which neural networks can numerically solve differential equations. We further this by providing three strategies that can be targeted to improve the efficacy of the solver. The first two -- sinusoidal representation networks \cite{SIREN} and random Fourier features \cite{FF} -- are well-established in the field of artificial neural networks and machine learning. The third is a novel technique called error correction \cite{AKSH1, AKSHetal, AKSH2, AKSH3, AKSH4}. We explain how error correction can be implemented recursively, with little modification to the original solver, to give enhanced numerical solutions to differential equations, and we present results that demonstrate this. 

This paper is designed to give a flavour of the competence of artificial neural networks in this field, while also highlighting their certain limitations.

\section{Background}

Throughout this paper, we consider differential equations with solution $\phi:\mathbb{R}^d\to\mathbb{R}$. Consequently, our neural network approximation is a function $\mathcal{N}:\mathbb{R}^d\to\mathbb{R}$.

\subsection{Universal approximation theorems}

The realisation of neural networks' capabilities to learn seemingly any function has brought about numerous universal approximation theorems. These state that, under certain conditions, neural networks are able to approximate any function to arbitrary closeness. We recall one of these theorems by Hornik \cite{Hornik1991}.

First, define the set of all functions represented by a neural network with a single hidden layer of width $n$ and identity activation on the output layer as

\begin{equation*}
    \mathscr{A}^n(\sigma) = \left\{\mathcal{N}:\mathbb{R}^d\to\mathbb{R}, \mathcal{N}(\mathbf{x})=\mathbf{W}^{(1)}\left(\sigma\left(\mathbf{W}^{(0)}\mathbf{x}+\mathbf{b}^{(0)}\right)\right)+b^{(1)}\right\}
\end{equation*}

where $\mathbf{x}\in\mathbb{R}^d, \mathbf{W}^{(0)}\in\mathbb{R}^{n\times d}, \mathbf{W}^{(1)}\in\mathbb{R}^{1\times n}, \mathbf{b}^{(0)}\in\mathbb{R}^n, b^{(1)}\in\mathbb{R},$ and $\sigma:\mathbb{R}\to\mathbb{R}$ is applied element-wise. Then, 

\begin{equation} \label{Fspace}
    \mathscr{A}(\sigma)=\bigcup_{n=1}^\infty\mathscr{A}^n(\sigma)
\end{equation}

is the set of all such functions with any number of neurons. Define also $\mathcal{C}^m(\mathbb{R}^d)$ as the space of all functions that, together with their partial derivatives of order $|\alpha|\leq m$, are continuous on $\mathbb{R}^d$.

\begin{theorem} \label{Hornik}
     \textup{\cite{Hornik1991}} If $\sigma\in\mathcal{C}^m(\mathbb{R}^d)$ is nonconstant and bounded, then $\mathscr{A}(\sigma)$ is uniformly $m$-dense on all compact sets of $\mathcal{C}^m(\mathbb{R}^d)$, i.e. for all $\phi\in\mathcal{C}^m(\mathbb{R}^d)$, for all compact sets $\Omega\subset\mathbb{R}^d$ and for all $\epsilon>0$, there exists $\mathcal{N}\in\mathscr{A}(\sigma)$ such that
    
    \begin{equation*}
        \max_{|\alpha|\leq m}\sup_{x\in\Omega} |\partial_x^{(\alpha)}\mathcal{N}(x) - \partial_x^{(\alpha)}\phi(x)|<\epsilon    
    \end{equation*}
    
\end{theorem}

Theorem \ref{Hornik} illustrates the universal approximation quality for single-layer networks of arbitrary width. Applying results from an earlier paper by Hornik et al. \cite{Hornik1989}, this can be extended to multilayer networks. Crucially, these theorems tell us that neural networks are dense on certain function spaces, but they do not tell us how to train a network to realise this.

\subsection{Neural network differential equation solvers}

Using neural networks to solve differential equations was introduced in the late 1990s \cite{NNDE_Pros}, but experienced a modern resurgence through the publication of two papers \cite{PINNS1, PINNS2} on physics-informed neural networks. The deep Galerkin method \cite{Sirignano_2018} which we describe below is very similar to the method described in \cite{PINNS1} only, instead of using experimental data, we train a network on points randomly sampled across the domain of the differential equation.

Consider Poisson's equation with Dirichlet boundary conditions:

\begin{equation} \label{Poisson}
    \begin{cases}
        \nabla^2\phi & = f \text{ in $\Omega$} \\
        \phi & = g \text{ on $\partial\Omega$}
    \end{cases}
\end{equation}

\begin{lemma}
    Let $\Omega\subset\mathbb{R}^d$ be a smooth, compact domain. Then there exists at most one solution $\phi$ to (\ref{Poisson}).  
\end{lemma}

\begin{proof}
    Suppose $\phi$ and $\varphi$ both satisfy the conditions of (\ref{Poisson}) and let $\omega=\phi - \varphi$. Then $\omega$ is harmonic in $\Omega$ and zero on $\partial\Omega$. Then,
    
    \begin{equation*}
        \int_\Omega \omega(\mathbf{x})\nabla^2\omega(\mathbf{x}) \,d\mathbf{x} = \int_{\partial\Omega} \omega(\mathbf{x})\delta_{\mathbf{n}}\omega(\mathbf{x}) \,d\mathbf{x} - \int_\Omega ||\nabla\omega(\mathbf{x})||^2 \,d\mathbf{x} = - \int_\Omega ||\nabla\omega(\mathbf{x})||^2 \,d\mathbf{x} = 0
    \end{equation*}
    
    and $\nabla\omega = 0$. Thus, $\omega=0$ and $\phi = \varphi$.
\end{proof}

We now seek an approximation $\mathcal{N}$ to $\phi$. Define the objective function

\begin{equation*}
    \mathcal{J}(\mathcal{N}) = \int_\Omega |\nabla^2\mathcal{N}(\mathbf{x}) - f(\mathbf{x})|^2 \nu_1(\mathbf{x}) \,d\mathbf{x} + \int_{\partial\Omega} |\mathcal{N}(\mathbf{x}) - g(\mathbf{x})|^2 \nu_2(\mathbf{x}) \,d\mathbf{x}
\end{equation*}

for probability distributions $\nu_1$ on $\Omega$ and $\nu_2$ on $\partial\Omega$. By uniqueness of $\phi$, $\mathcal{J}(\mathcal{N})=0\implies\mathcal{N}=\phi$. However, minimising the objective function directly is impractical. First, we transform the problem into a machine learning framework. Our approximation $\mathcal{N}=\mathcal{N}(\cdot;\theta)$ becomes a neural network with parameters $\theta$.

\subsubsection*{Deep Galerkin method}

We demonstrate the algorithm for the deep Galerkin method \cite{Sirignano_2018} when applied to Poisson's equation (\ref{Poisson}):

\begin{enumerate}
    \item Randomly sample points $\{\mathbf{x}_i\}_{i=1}^M$ from $\Omega$ and $\{\mathbf{y}_j\}_{j=1}^N$ from $\partial\Omega$ according to respective probability distributions $\nu_1$ and $\nu_2$, and propagate them through a feed-forward neural network $\mathcal{N}(\cdot;\theta)$.
    \item Calculate the loss:
    \[\mathcal{L}(\theta)=\frac{1}{M}\sum_{i=1}^M \left(\nabla^2\mathcal{N}(\mathbf{x}_i;\theta) - f(\mathbf{x}_i)\right)^2 + \frac{1}{N}\sum_{j=1}^N \left(\mathcal{N}(\mathbf{y}_j;\theta) - g(\mathbf{y}_j)\right)^2\]
    \item Update parameters $\theta_{t+1}=\theta_t - \eta\nabla_\theta\mathcal{L}(\theta_t)$ with learning rate $\eta>0$ and $t\in\mathbb{N}_0$.
    \item Repeat until $\nabla_\theta\mathcal{L}(\theta_t)\approx0$.
\end{enumerate}

This is a minibatch gradient descent implementation, where $M$ and $N$ are the size of the minibatches and $M>N$.

\begin{lemma}
    $\mathbb{E}[\nabla_\theta\mathcal{L}(\theta_t)|\theta_t] = \nabla_\theta\mathcal{J}(\mathcal{N}(\cdot;\theta_t))$    
\end{lemma}

\begin{proof}
    Assume $\mathcal{L}$ sufficiently smooth and bounded to interchange derivatives and integrals. Then,
    \begin{align*}
        \mathbb{E}[\nabla_\theta\mathcal{L}(\theta_t)|\theta_t] &= \nabla_\theta\left[\frac{1}{M}\sum_{i=1}^M \mathbb{E}\left[(\nabla^2\mathcal{N}(\mathbf{x}_i;\theta_t) - f(\mathbf{x}_i))^2\right] + \frac{1}{N}\sum_{j=1}^N \mathbb{E}\left[(\mathcal{N}(\mathbf{y}_j;\theta_t) - g(\mathbf{y}_j))^2\right]\right] \\
        &= \nabla_\theta\left[\frac{1}{M}\sum_{i=1}^M \int_\Omega (\nabla^2\mathcal{N}(\mathbf{x};\theta_t) - f(\mathbf{x}))^2 \nu_1(\mathbf{x}) \,d\mathbf{x} + \frac{1}{N}\sum_{j=1}^N \int_{\partial\Omega} (\mathcal{N}(\mathbf{y};\theta_t) - g(\mathbf{y}))^2 \nu_2(\mathbf{y}) \,d\mathbf{y}\right] \\
        &= \nabla_\theta\left[\int_\Omega (\nabla^2\mathcal{N}(\mathbf{x};\theta_t) - f(\mathbf{x}))^2 \nu_1(\mathbf{x}) \,d\mathbf{x} + \int_{\partial\Omega} (\mathcal{N}(\mathbf{y};\theta_t) - g(\mathbf{y}))^2 \nu_2(\mathbf{y}) \,d\mathbf{y}\right] \\
        &= \nabla_\theta\mathcal{J}(\mathcal{N}(\cdot;\theta_t))
    \end{align*}
\end{proof}

Therefore, the $\nabla_\theta\mathcal{L}(\theta_t)$ are unbiased estimates of $\nabla_\theta\mathcal{J}(\mathcal{N}(\cdot;\theta_t))$, and we can assume a step in the descent direction of $\mathcal{L}$ is also one in $\mathcal{J}$. Thus, any minimisation of $\mathcal{L}$ should translate to a local minimisation of $\mathcal{J}$.

\subsubsection*{Minimisation of $\mathcal{J}(\mathcal{N})$}

We prove the following theorem, adapted from the original deep Galerkin method paper \cite{Sirignano_2018}.

\begin{theorem} \label{bigboy}
    Let $\mathscr{A}(\sigma)$ be given by (\ref{Fspace}), for nonconstant, bounded $\sigma$, and let $\Omega\in\mathbb{R}^d$ be a compact domain and consider measures $\nu_1, \nu_2$ whose supports are contained in $\Omega,\partial\Omega$ respectively. Assume further that $\nabla^2\phi$ is locally Lipschitz with Lipschitz constant that can have at most polynomial growth on $\nabla\phi$, uniformly with respect to $x$, i.e.
    
    \begin{equation} \label{Lippy}
        |\nabla^2\mathcal{N}-\nabla^2\phi|\leq\left(||\nabla\mathcal{N}||^{\frac{a}{2}}+||\nabla\phi||^{\frac{b}{2}}\right)||\nabla\mathcal{N}-\nabla\phi||
    \end{equation}
    
    for some constants $0\leq a,b<\infty$. Then, for all $\epsilon>0$, there exists a constant $\kappa>0$ such that there exists a function $\mathcal{N}\in\mathscr{A}(\sigma)$ with
    
    \begin{equation*}
        \mathcal{J}(\mathcal{N})\leq\kappa\epsilon
    \end{equation*}
    
\end{theorem}

\begin{proof}
    The condition given by (\ref{Lippy}) implies that
    
    \begin{align*}
        |\nabla^2\mathcal{N}-\nabla^2\phi|^2 &\leq \left(||\nabla\mathcal{N}||^{\frac{a}{2}}+||\nabla\phi||^{\frac{b}{2}}\right)^2||\nabla\mathcal{N}-\nabla\phi||^2 \\
        &\leq \left(||\nabla\mathcal{N}||^a+||\nabla\phi||^b+2||\nabla\mathcal{N}||^{\frac{a}{2}}||\nabla\phi||^{\frac{b}{2}}\right)||\nabla\mathcal{N}-\nabla\phi||^2 \\
        &\leq 2\left(||\nabla\mathcal{N}||^a+||\nabla\phi||^b\right)||\nabla\mathcal{N}-\nabla\phi||^2
    \end{align*}
    
    with the last line following from Young's inequality \cite{inequalities}. Then,
    
    \begin{align*}
        \int_\Omega |\nabla^2\mathcal{N}(\mathbf{x}) - \nabla^2\phi(\mathbf{x})|^2 \,d\nu_1(\mathbf{x}) &\leq 2\int_\Omega \left(||\nabla\mathcal{N}(\mathbf{x})||^a+||\nabla\phi(\mathbf{x})||^b\right)||\nabla\mathcal{N}(\mathbf{x})-\nabla\phi(\mathbf{x})||^2 \,d\nu_1(\mathbf{x}) \\
        &\leq 2\left[\int_\Omega \left(||\nabla\mathcal{N}(\mathbf{x})||^a+||\nabla\phi(\mathbf{x})||^b\right)^p\,d\nu_1(\mathbf{x})\right]^{\frac{1}{p}}\left[\int_\Omega||\nabla\mathcal{N}(\mathbf{x})-\nabla\phi(\mathbf{x})||^{2q}\,d\nu_1(\mathbf{x})\right]^{\frac{1}{q}}
    \end{align*}
    
    if we apply H\"{o}lder's inequality \cite{inequalities} for exponents $p,q$ satisfying $\frac{1}{p}+\frac{1}{q}=1$ and $1\leq p,q\leq\infty$. Furthermore, 
    
    \begin{align*}
        \int_\Omega |\nabla^2\mathcal{N}(\mathbf{x}) - \nabla^2\phi(\mathbf{x})|^2 \,d\nu_1(\mathbf{x}) &\leq K\left[\int_\Omega \left(||\nabla\mathcal{N}(\mathbf{x})-\nabla\phi(\mathbf{x})||^a+||\nabla\phi(\mathbf{x})||^{\max\{a,b\}}\right)^p\,d\nu_1(\mathbf{x})\right]^{\frac{1}{p}} \\ &\quad \cdot\left[\int_\Omega||\nabla\mathcal{N}(\mathbf{x})-\nabla\phi(\mathbf{x})||^{2q}\,d\nu_1(\mathbf{x})\right]^{\frac{1}{q}} \\
        &\leq K(\epsilon^a + \sup_{\mathbf{x}\in\Omega} ||\nabla\phi(\mathbf{x})||^{\max\{a,b\}})\epsilon^2 
    \end{align*}
    
    for some constant $K$. The last line follows from Theorem \ref{Hornik}. Applying this result and Theorem \ref{Hornik} again to the objective function $\mathcal{J}$, we obtain:
    
    \begin{align*}
        \mathcal{J}(\mathcal{N}) &= \int_\Omega |\nabla^2\mathcal{N}(\mathbf{x}) - f(\mathbf{x})|^2 \,d\nu_1(\mathbf{x}) + \int_{\partial\Omega} |\mathcal{N}(\mathbf{x}) - g(\mathbf{x})|^2 \,d \nu_2(\mathbf{x}) \\
        &= \int_\Omega |\nabla^2\mathcal{N}(\mathbf{x}) - \nabla^2\phi(\mathbf{x})|^2 \,d\nu_1(\mathbf{x}) + \int_{\partial\Omega} |\mathcal{N}(\mathbf{x}) - \phi(\mathbf{x})|^2 \,d \nu_2(\mathbf{x}) \\
        &\leq K(\epsilon^a + \sup_{\mathbf{x}\in\Omega} ||\nabla\phi(\mathbf{x})||^{\max\{a,b\}})\epsilon^2 + \epsilon^2
    \end{align*}
    
    Finally, a rescaling of $\epsilon>0$ yields
    
    \begin{equation*}
        \mathcal{J}(\mathcal{N})\leq\kappa\epsilon
    \end{equation*}
    
    for some constant $\kappa>0$ which may depend on $\sup\limits_{\mathbf{x}\in\Omega} ||\nabla\phi(\mathbf{x})||$.
   
\end{proof}

Theorem \ref{bigboy} guarantees the existence of a feed-forward neural network $\mathcal{N}$ that, under relatively relaxed conditions, makes the objective function $\mathcal{J}(\mathcal{N})$ for Poisson's equation arbitrarily small. However, neural network objective functions are highly non-convex. This means they have numerous minima and, while gradient descent algorithms like the deep Galerkin method are extremely effective at reaching said minima \cite{GradDescent}, there is no guarantee of achieving the global minimum i.e., in our case, finding the unique solution. Many authors research such ideas in non-convex optimisation \cite{DasBlog}, but we do not touch on them here, and present only empirical evidence of our solver finding/not finding global minima in the Results section (see \textbf{\ref{Results}}).

\section{Methods}

We now present three highly accessible methods to enhance the performance of a neural network trained to solve differential equations via the deep Galerkin method.

\subsection{Sinusoidal representation networks}

Consider a neural network that is trained to approximate a function directly. We need only the first-order derivatives of the activation functions to backpropagate, and thus ReLU seems a natural choice \cite{ReLU}. However, our framework requires a network to learn a function via its derivatives. ReLU networks cannot do this without significant loss of information since they have second derivative zero. They are incapable of accurately modelling a signal's higher-order derivatives.

A recent paper \cite{SIREN} highlighting these limitations proposes something the authors call a sinusoidal representation network or SIREN. This is a neural network that implicitly defines a function, in our case $\mathcal{N}$, with sinusoidal activations. Thus, while regular feed-forward networks with, say, ReLU activation may be excellent function approximators, a SIREN can further accurately fit derivatives of functions $\phi$ through its own derivatives. ReLU networks typically cannot, due to their piecewise linear nature. This idea is hidden in Theorem \ref{Hornik} since ReLU is continuous but not differentiable, and so a network $\mathcal{N}$ with ReLU activation could only achieve

\begin{equation*}
    \sup_x |\partial_x^{(\alpha)}\mathcal{N}(x) - \partial_x^{(\alpha)}\phi(x)|<\epsilon     
\end{equation*}

for $\alpha=0$. By contrast, $\sin\in\mathcal{C}^\infty$, so the equivalent statement is true for any $|\alpha|<\infty$.

Evaluating the gradient of a SIREN scales quadratically in the number of layers of the SIREN \cite{SIREN}. So, fitting higher-order derivatives is no easy task. However, for simple differential equations like Poisson's equation, it is computationally feasible, and the authors of \cite{SIREN} provide experimental results that show SIRENs are excellent at modelling first and second-order derivatives of complicated signals, as well as the high-frequency signals themselves.

\subsection{Random Fourier features}

Recent works \cite{freqbias, specbias} have described a spectral bias inherent to neural networks learning functions. They prioritise learning the low-frequency modes of the functions and thus, high frequencies are captured much later in the training procedure. 

In many ways, this is a key reason behind the immense success of neural networks. Often, they are over-parameterised, i.e. the number of parameters far exceeds the number of training samples yet, counter-intuitively, they still show remarkable capacity to generalise well \cite{overparam}. Spectral bias may explain part of this phenomenon because it suggests, if there is a way to fit data effectively with only low frequencies, then a neural network will do just this, without needing to resort to high frequencies that overfit the data.

However, this also means that neural networks struggle to learn high frequency functions. Theoretical results in \cite{freqbiasconvrate} show that a one-dimensional function of pure frequency $\omega$, e.g. $\cos(\omega x)$, is learned in time that scales with $\omega^2$. This is ratified experimentally.

A 2020 paper \cite{FF} publishes results on the use of a Fourier feature mapping to effectively overcome this spectral bias, and allow multilayer perceptrons (MLPs) to learn high frequency functions in low-dimensional domains. The authors motivate such work with neural tangent kernel (NTK) theory. NTKs have been shown to model the behaviour of MLPs in the infinite-width limit during training \cite{NTK}. We do not describe them in detail here, but give a summary of the main idea behind Fourier feature mapping. For two different inputs $\mathbf{x},\mathbf{x'}$ to the MLP, the corresponding NTK can be given by

\newpage

\begin{equation*}
    NTK(\mathbf{x},\mathbf{x'})=h(\mathbf{x}^T\mathbf{x'})
\end{equation*}

where $h$ is some scalar function \cite{FF}.

The mapping

\begin{equation} \label{FFmap}
    \gamma(\mathbf{x}) = [\cos(2\pi\mathbf{Bx}), \sin(2\pi\mathbf{Bx})]^T
\end{equation}

is a Gaussian random Fourier feature mapping for $\mathbf{x}\in\mathbb{R}^d$, where each entry in $\mathbf{B}\in\mathbb{R}^{n\times d}$ is sampled from a normal distribution with mean zero and variance $\Sigma^2$. Therefore,

\begin{align*}
    NTK(\gamma(\mathbf{x}), \gamma(\mathbf{x'})) &= h(\gamma(\mathbf{x})^T\gamma(\mathbf{x'})) \\ &= h\left(\cos(2\pi\mathbf{Bx})\cos(2\pi\mathbf{Bx'})+\sin(2\pi\mathbf{Bx})\sin(2\pi\mathbf{Bx'})\right) \\
    &= h(\cos(2\pi\mathbf{B}(\mathbf{x}-\mathbf{x'})))
\end{align*}

Crucially, this defines a kernel function with width controlled by the random matrix $\mathbf{B}$. Kernel functions are used to fit data, and their width directly influences whether they overfit (with high frequencies) or underfit (with low frequencies). So, given that this function characterises the evolution of the MLP during training, we can tune the network towards learning particular frequencies by simply changing $\Sigma$:

\begin{itemize}
    \item A small $\Sigma$ gives a wide kernel that will underfit a high-frequency function.
    \item A large $\Sigma$ gives a narrow kernel that will overfit a low-frequency function.
\end{itemize}

In our framework, $\Sigma$ is now just another hyperparameter, and we can find the optimal $\Sigma$ through a simple sweep of values. We choose the value that gives the fastest convergence. The authors of \cite{FF} also advise that $n$, the number of Fourier features, improves performance with size. Of course, there is a computational cost associated with increasing $n$, so it is best taken `as small as gives good enough results.'

\subsection{Error correction}

We introduce the main work of this paper; the novel technique error correction \cite{AKSH1, AKSHetal, AKSH2, AKSH3, AKSH4} is designed to increase the efficacy of any neural network differential equation solver. This method is general and can be applied to all differential equations, in combination with any such similar strategies, such as Koopman boosting \cite{AKSHwtr} or those presented above. Much of the work here was proposed in \cite{AKSH1} and formalised in \cite{AKSHetal}, which the reader should refer to as supplement.

When dealing with neural networks, we bank on the idea that a `small enough' loss implies a `good enough' accuracy. Now, in many scenarios, this ideology fails because zero loss would represent drastic overfitting. Conveniently, this does not concern us as we want our network to fit the (training) data as accurately as possible. Still, the original problem remains; how can we know how close we are to the true solution $\phi$?

It turns out analysis and estimation of the unknown error between $\phi$ and $\mathcal{N}$ is possible. Indeed, in \cite{AKSH2}, the author shows how you can obtain specific bounds on this error, without knowledge of $\phi$. In this section, we provide a correction method (based on this error) to enhance neural network differential equation solvers, by overcoming performance saturation when the network settles around a local minimum of the loss function. Here, we also make use of differential equation operators which send true solutions to zero. Consider this for Poisson's equation (\ref{Poisson}):

\begin{equation*}
    \mathbf{F}[\cdot] = \nabla^2[\cdot] - f
\end{equation*}

\newpage

Define $\phi_\epsilon=\phi - \mathcal{N}$ as the error between the unknown solution $\phi$ and a fixed approximation $\mathcal{N}$. Clearly, 

\begin{align*}
    \mathbf{F}[\mathcal{N}]&=\nabla^2\mathcal{N} - f \\
    &= \nabla^2[\phi-\phi_\epsilon] - f \\
    &= \nabla^2\phi - f - \nabla^2\phi_\epsilon \\
    &= -\nabla^2\phi_\epsilon
\end{align*}

since $\mathbf{F}[\phi]=\nabla^2\phi-f=0$. Thus, $\mathbf{F}[\mathcal{N}]+\nabla^2\phi_\epsilon = 0$ and, given that $\mathbf{F}[\mathcal{N}]$ is completely independent to $\phi_\epsilon$, we have defined a new Poisson's equation. Our general strategy now will be to train a neural network $\mathcal{N}_\epsilon$ to approximate $\phi_\epsilon$ through the conditions of this new differential equation. Then, $\mathcal{N}+\mathcal{N_\epsilon}\approx\mathcal{N}+\phi_\epsilon=\phi$.

Before we formalise and evaluate this method, note that it applies also to differential equations with non-linear terms. Consider the Poisson-Boltzmann equation with Dirichlet boundary conditions:

\begin{equation*}
    \begin{cases}
        \nabla^2\phi + \sinh\phi & = f \text{ in $\Omega$} \\
        \phi & = g \text{ on $\partial\Omega$}
    \end{cases}
\end{equation*}

Define the operator

\begin{equation*}
    \mathbf{G}[\cdot] = \nabla^2[\cdot]+\sinh[\cdot]-f
\end{equation*}

and, once again, have $\phi_\epsilon=\phi - \mathcal{N}$. Then,

\begin{align*}
    \mathbf{G}[\mathcal{N}] &= \nabla^2\mathcal{N}+\sinh\mathcal{N}-f \\
    &= \nabla^2[\phi-\phi_\epsilon]+\sinh\mathcal{N} +\sinh\phi - \sinh\phi - f \\
    &= \nabla^2\phi + \sinh\phi - f -\nabla^2\phi_\epsilon + \sinh\mathcal{N} - \sinh\phi \\
    &= -\nabla^2\phi_\epsilon + \sinh\mathcal{N} - \sinh(\mathcal{N}+\phi_\epsilon)
\end{align*}

since $\mathbf{G}[\phi]=\nabla^2\phi + \sinh\phi - f=0$. A clever trick of adding and subtracting $\sinh\phi$ allows the $\mathbf{G}[\phi]$ term to be removed from the equation. In the last line, we simply seek to keep the equation explicit in $\mathcal{N}$ and $\phi_\epsilon$.

\subsubsection*{Theoretical results}

Now, we formalise this idea of error correction, adapting the approach from \cite{AKSHetal}. Consider a differential equation over $\Omega$ in operator form:

\begin{equation} \label{F0}
    \mathbf{F_0}[\phi] = \mathbf{A}[\phi] + \mathbf{B}[\phi] + \mathbf{C} = 0
\end{equation}

where $\mathbf{A}$ represents the terms that depend linearly on $\phi$, $\mathbf{B}$ represents those that depend non-linearly on $\phi$, and $\mathbf{C}$ is independent of $\phi$. The solution $\phi$ may also admit some constraints on the boundary $\partial\Omega$ but, for now, these are not of interest. Assume also that $\phi$ is unique.

We first prove a result that follows from the inverse function theorem \cite{IVT}:

\begin{theorem} \label{IVT}
    \textup{(Inverse function theorem).} Suppose that $F:\mathbb{R}^n\to\mathbb{R}^n$ is continuously differentiable in some open set containing $x^*$, and suppose moreover that the Jacobian $DF(x^*)$ is invertible. Then there exists open sets $U,V\subset\mathbb{R}^n$ with $x^*\in U$ and $F(x^*)\in V$ such that $F:U\to V$ is a bijection, and $F^{-1}:V\to U$ is continuously differentiable for all $y\in V$ with
    
    \[DF^{-1}(y)=\left[DF(F^{-1}(y))\right]^{-1}\]
    
\end{theorem}

\begin{corollary} \label{cor1}
    Suppose that $\mathbf{F_0}:\mathbb{R}\to\mathbb{R}$ in (\ref{F0}) is continuously differentiable in some open set containing $\phi^*$, that $D\mathbf{F_0}[\phi^*]$ is invertible, and $\mathbf{F_0}[\phi^*]=0$. Then, there is a neighbourhood of $0$ small enough such that
    
    \[\mathbf{F_0}[\mathcal{N}]\to 0\implies \mathcal{N}\to\phi^*\]
    
\end{corollary}

\begin{proof}
    By Theorem \ref{IVT}, choose neighbourhoods $U,V\subset\mathbb{R}$ with $\phi^*\in U,0\in V$ such that $\mathbf{F_0}:U\to V$ is a bijection and $\mathbf{F_0}^{-1}:V\to U$ is continuous differentiable for all $y\in V$. For $\mathcal{N}\in U$, the continuity of $\mathbf{F_0}^{-1}$ implies that
    
    \[\mathbf{F_0}[\mathcal{N}]\to 0\implies \mathcal{N}\to\phi^*\]
    
\end{proof}

Thus, assuming we can minimise the loss function for some neural network $\mathcal{N}$ such that $\mathbf{F_0}[\mathcal{N}]\to 0$ at all points, then $\mathcal{N}\to\phi$ at all points. So, let us train such a network $\mathcal{N}_0$ to approximate $\phi$ via (\ref{F0}). Define also $\phi_1 = \phi - \mathcal{N}_0$.

\begin{align*}
    \mathbf{F_0}[\mathcal{N}_0] &= \mathbf{A}[\mathcal{N}_0] + \mathbf{B}[\mathcal{N}_0] + \mathbf{C} \\
    &= \mathbf{A}[\phi-\phi_1] + \mathbf{B}[\mathcal{N}_0] + \mathbf{B}[\phi] - \mathbf{B}[\phi] + \mathbf{C}  \\
    &= \mathbf{A}[\phi] + \mathbf{B}[\phi] + \mathbf{C} - \mathbf{A}[\phi_1] + \mathbf{B}[\mathcal{N}_0] - \mathbf{B}[\phi] \\
    &= - \mathbf{A}[\phi_1] + \mathbf{B}[\mathcal{N}_0] - \mathbf{B}[\mathcal{N}_0+\phi_1]
\end{align*}

since $\mathbf{F_0}[\phi]=\mathbf{A}[\phi] + \mathbf{B}[\phi] + \mathbf{C}=0$ by definition. We have defined a new differential equation in operator form:

\begin{equation*}
    \mathbf{F_1}[\phi_1] = \mathbf{F_0}[\mathcal{N}_0] + \mathbf{A}[\phi_1] - \mathbf{B}[\mathcal{N}_0] + \mathbf{B}[\mathcal{N}_0+\phi_1] = 0
\end{equation*}

$\phi_1$ solves the above equation exactly and, given the uniqueness of $\phi$, is also unique. Now, train some other neural network $\mathcal{N}_1$ to approximate $\phi_1$, and define $\phi_2 = \phi_1-\mathcal{N}_1$. Once again,

\begin{align*}
    \mathbf{F_1}[\mathcal{N}_1] &= \mathbf{F_0}[\mathcal{N}_0] + \mathbf{A}[\mathcal{N}_1] - \mathbf{B}[\mathcal{N}_0] + \mathbf{B}[\mathcal{N}_0+\mathcal{N}_1] \\
    &= \mathbf{F_0}[\mathcal{N}_0] + \mathbf{A}[\phi_1-\phi_2]-\mathbf{B}[\mathcal{N}_0]+\mathbf{B}[\mathcal{N}_0+\mathcal{N}_1]+\mathbf{B}[\mathcal{N}_0+\phi_1]-\mathbf{B}[\mathcal{N}_0+\phi_1] \\
    &= \mathbf{F_0}[\mathcal{N}_0] + \mathbf{A}[\phi_1] - \mathbf{B}[\mathcal{N}_0] + \mathbf{B}[\mathcal{N}_0+\phi_1] - \mathbf{A}[\phi_2] + \mathbf{B}[\mathcal{N}_0+\mathcal{N}_1] - \mathbf{B}[\mathcal{N}+\phi_1] \\
    &= - \mathbf{A}[\phi_2] + \mathbf{B}[\mathcal{N}_0+\mathcal{N}_1] - \mathbf{B}[\mathcal{N}_0+\phi_1] \\
    &= - \mathbf{A}[\phi_2] + \mathbf{B}[\mathcal{N}_0+\mathcal{N}_1] - \mathbf{B}[\mathcal{N}_0+\mathcal{N}_1 + \phi_2]
\end{align*}

since $\mathbf{F_1}[\phi_1]=\mathbf{F_0}[\mathcal{N}_0] + \mathbf{A}[\phi_1] - \mathbf{B}[\mathcal{N}_0] + \mathbf{B}[\mathcal{N}_0+\phi_1]=0$, and $\phi_1=\mathcal{N}_1 + \phi_2$. We define a further differential equation in operator form:

\begin{equation*}
    \mathbf{F_2}[\phi_2] = \mathbf{F_1}[\mathcal{N}_1] + \mathbf{A}[\phi_2] - \mathbf{B}[\mathcal{N}_0+\mathcal{N}_1] + \mathbf{B}[\mathcal{N}_0+\mathcal{N}_1 + \phi_2] = 0
\end{equation*}

Now, repeat the process. This algorithm can continue indefinitely, and we summarise the steps below. The idea is that our error-corrected approximation $\mathcal{N}_0+\mathcal{N}_1+\mathcal{N}_2+...$ will be more accurate than the once-trained approximation $\mathcal{N}_0$. This strategy is not unseen in the field of numerical methods to differential equations, we just apply it here to neural network solvers.

Let us define a recursive differential equation for the $k\textsuperscript{th}$ error correction. At this point, we have trained the initial network $\mathcal{N}_0$, and also a further $k-1$ residual networks $\mathcal{N}_1,\mathcal{N}_2,...,\mathcal{N}_{k-1}$. Our current error-corrected approximation is $\mathcal{N}^{(k-1)}=\mathcal{N}_0+\mathcal{N}_1+\mathcal{N}_2+...+\mathcal{N}_{k-1}$. Define $\phi_k=\phi_{k-1}-\mathcal{N}_{k-1}$. Now, train a new network $\mathcal{N}_k$ to approximate $\phi_k$ through the following differential equation:

\begin{equation} \label{Fk}
    \mathbf{F_k}[\phi_k] = \mathbf{F_{k-1}}[\mathcal{N}_{k-1}] + \mathbf{A}[\phi_k] - \mathbf{B}[\mathcal{N}^{(k-1)}] + \mathbf{B}[\mathcal{N}^{(k-1)} + \phi_k] = 0
\end{equation}

\begin{remark}
    $\mathbf{F_k}[\mathcal{N}_k]\equiv\mathbf{F_0}[\mathcal{N}^{(k)}]$.   
\end{remark}

\begin{corollary} \label{cor2}
    Suppose that $\mathbf{F_k}:\mathbb{R}\to\mathbb{R}$ in (\ref{Fk}) is continuously differentiable in some open set containing $\phi_k^*$, that $D\mathbf{F_k}[\phi_k^*]$ is invertible, and $\mathbf{F_k}[\phi_k^*]=0$. Then, there is a neighbourhood of $0$ small enough such that
    
    \[\mathbf{F_k}[\mathcal{N}_k]\to 0\implies \mathcal{N}_k\to\phi_k^*\]
    
    Furthermore,
    
    \[|\mathcal{N}^{(k)}-\phi|=\mathcal{O}\left(|\mathbf{F_k}[\mathcal{N}_k]|\right)\]
    
\end{corollary}

\begin{proof}
    The first result follows analogously from the inverse function theorem as in Corollary \ref{cor1}.
    
    By Theorem \ref{IVT}, $\mathbf{F_0}^{-1}$ is continuously differentiable on some open set around $0$. Thus, it is also locally Lipschitz continuous around $0$, meaning there exists some constant $\alpha\geq 0$ such that
    
    \begin{align*}
        |\mathcal{N}^{(k)}-\phi| &= |\mathbf{F_0}^{-1}[\mathbf{F_0}[\mathcal{N}^{(k)}] - \mathbf{F_0}^{-1}[\mathbf{F_0}[\phi]]| \\
        &\leq \alpha|\mathbf{F_0}[\mathcal{N}^{(k)}]-\mathbf{F_0}[\phi]| \\
        &\leq \alpha|\mathbf{F_0}[\mathcal{N}^{(k)}]| \\
        &\leq \alpha|\mathbf{F_k}[\mathcal{N}_k]|
    \end{align*}
    
    and therefore,
    
    \[|\mathcal{N}^{(k)}-\phi|=\mathcal{O}\left(|\mathbf{F_k}[\mathcal{N}_k]|\right)\]
    
\end{proof}

Finally, given Dirichlet boundary conditions $\phi=g$ on $\partial\Omega$, any $\phi_k$ is known exactly over $\partial\Omega$ since $\phi_k=\phi-\mathcal{N}^{(k-1)}$. Thus, the loss function for the $k\textsuperscript{th}$ error correction can be defined as

\begin{equation} \label{Lk}
    \mathcal{L}_k(\theta^{(k)})=\frac{1}{M}\sum_{i=1}^M \left(\mathbf{F_k}\left[\mathcal{N}_k\left(\mathbf{x}_i;\theta^{(k)}\right)\right]\right)^2 + \frac{1}{N}\sum_{j=1}^N \left(\mathcal{N}_k\left(\mathbf{y}_j;\theta^{(k)}\right)-\phi_k(\mathbf{y}_j)\right)^2
\end{equation}

for some randomly sampled points $\{\mathbf{x}_i\}_{i=1}^M$ from $\Omega$ and $\{\mathbf{y}_j\}_{j=1}^N$ from $\partial\Omega$. 

\subsubsection*{Algorithm}

The error correction algorithm to order $K$ proceeds as follows:

\begin{enumerate}
    \item Train a neural network $\mathcal{N}_0$ to satisfy the conditions of a differential equation given by (\ref{F0}) and constraint conditions. Once the loss has converged, stop training and freeze the parameters of $\mathcal{N}_0$.
    \item Initiate and train new neural networks $\{\mathcal{N}_k\}_{k=1}^K$ in sequence to satisfy differential equations given by (\ref{Fk}), via loss functions (\ref{Lk}). Once the loss has converged, stop training, freeze the parameters of $\mathcal{N}_k$, and proceed with $\mathcal{N}_{k+1}$. 
    \item The solution to (\ref{F0}) is approximated by $\mathcal{N}:=\mathcal{N}^{(K)}=\sum\limits_{k=0}^K \mathcal{N}_k$.
\end{enumerate}

This is given above for Dirichlet boundary conditions, but works generally if you incorporate the constraint conditions into all loss functions.

\subsubsection*{Poisson's equation}

For Poisson's equation (\ref{Poisson}), $\mathbf{B}\equiv 0$ since the Laplacian is linear, so we can write the $k\textsuperscript{th}$ differential equation as

\begin{equation*}
    \mathbf{F_k}[\phi_k] = \mathbf{F_{k-1}}[\mathcal{N}_{k-1}] + \nabla^2[\phi_k] = 0    
\end{equation*}

which is a Poisson's equation with our usual $f=-\mathbf{F_{k-1}}[\mathcal{N}_{k-1}]$. Thus, we can apply Theorem \ref{bigboy} to guarantee that there are neural networks out there that can get very, very close to the $\phi_k$. In the next section, we provide evidence that shows, if we can train just two or three of these networks to reasonably approximate their true solutions, our error-corrected approximation will be a more accurate numerical solution to the original differential equation.

\section{Results} \label{Results}

We present results for a variety of different Poisson's equations (\ref{Poisson}). Our choice of Poisson's equation is motivated by its immense application in many areas of theoretical physics, including electrostatics and fluid dynamics. It is also the simplest second-order, linear PDE, making for a concise yet insightful demonstration of the power of error correction in neural network differential equation solvers.

To achieve this, we choose the function $f$ on the RHS to force a particular solution $\phi$ that we want to capture. For example, $f(x)=1$ would force the solution $\phi(x)=\frac{1}{2}x^2+c_1x+c_0$. However, in general, $f$ can be anything, particularly something which does not admit a closed-form solution to (\ref{Poisson}), and we do this for ease of visualising $\phi$.

Knowing the ground truth solution $\phi$ in closed form also allows us to compute the relative error 

\begin{equation*}
    \frac{\sum\limits_{\mathbf{x}\in S} \left(\phi(\mathbf{x})-\mathcal{N}(\mathbf{x})\right)^2}{\sum\limits_{\mathbf{x}\in S} \phi(\mathbf{x})^2}
\end{equation*}

at each epoch (iteration) of the training procedure, so we have an understanding of the success of our solver. It is important to note that, while we know $\phi$ and the relative error associated with our approximation, the neural network does not, and is solely trained via the loss function.

All neural networks used are SIRENs with 5 hidden layers and 128 hidden units per layer. They are trained on batches of 256, using the stochastic gradient descent variant Adam \cite{ADAM}, and learning rates are manually tuned for each case of Poisson's equation. All experiments are run on a 1.8 GHz Dual-Core Intel Core i5 CPU.

\newpage

\subsection{3D Poisson's Equation}

\begin{figure}[h]
\centering
\includegraphics[width=0.5\textwidth, trim = {24.5cm 0 0 0}, clip]{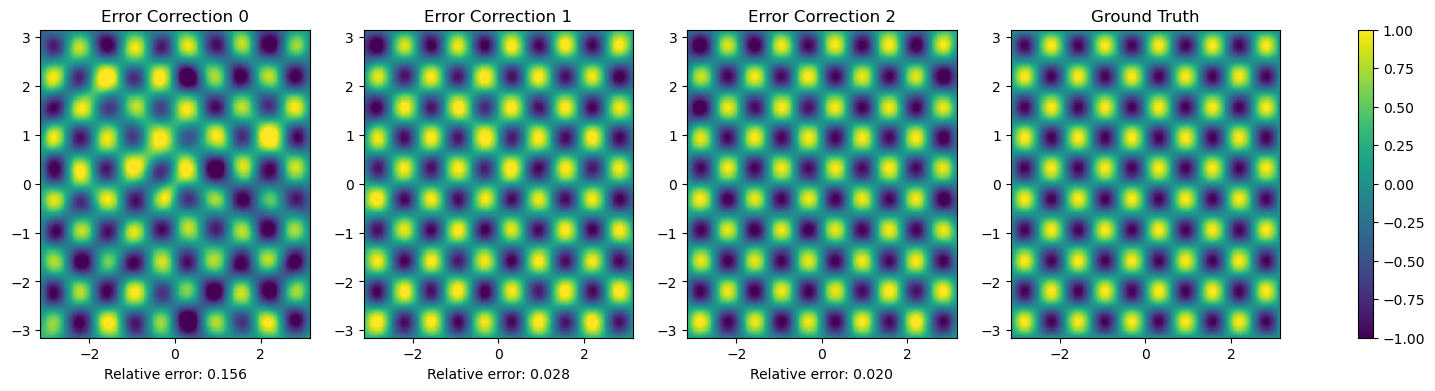}
\caption{$\phi(x,y,z)=\sin(5x)\sin(5y)\sin(5z)$ at $z=\frac{\pi}{10}$} \label{phi1}
\end{figure}

Figure \ref{phi1} shows the solution to Poisson's equation:

\begin{equation} \label{p1}
    \begin{cases}
        \nabla^2\phi & = -75\sin(5x)\sin(5y)\sin(5z) \text{ in $\Omega=[-\pi,\pi]\times[-\pi,\pi]\times[-\pi,\pi]$} \\
        \phi & =  0\text{ on $\partial\Omega$}
    \end{cases}
\end{equation}

Figure \ref{N1} shows our numerical solutions, with $\mathcal{N}^{(0)}=\mathcal{N}_0$ on the left, $\mathcal{N}^{(1)}=\mathcal{N}_0+\mathcal{N}_1$ in the centre, and $\mathcal{N}^{(2)}=\mathcal{N}_0+\mathcal{N}_1+\mathcal{N}_2$ on the right. We refer to these as Error Correction 0, 1 and 2, respectively.

Visually, all error corrections seem to capture the solution well. Furthermore, each correction decreases the relative error (printed at the bottom of of Figure \ref{N1}). Error Correction 1 does so significantly, while the improvement in accuracy from Error Correction 2 is marginal.

This is further captured in Figure \ref{N1m}, which is a plot of the loss and relative error per epoch. After finding a local minimum in Error Correction 0, the loss fluctuates erratically until we initialise Error Correction 1. The improvement is truly appreciable, and felt across the trends in relative error too.

\newpage

\begin{figure}[h!]
\centering
\includegraphics[width=\textwidth, trim = {0 0 11.5cm 0}, clip]{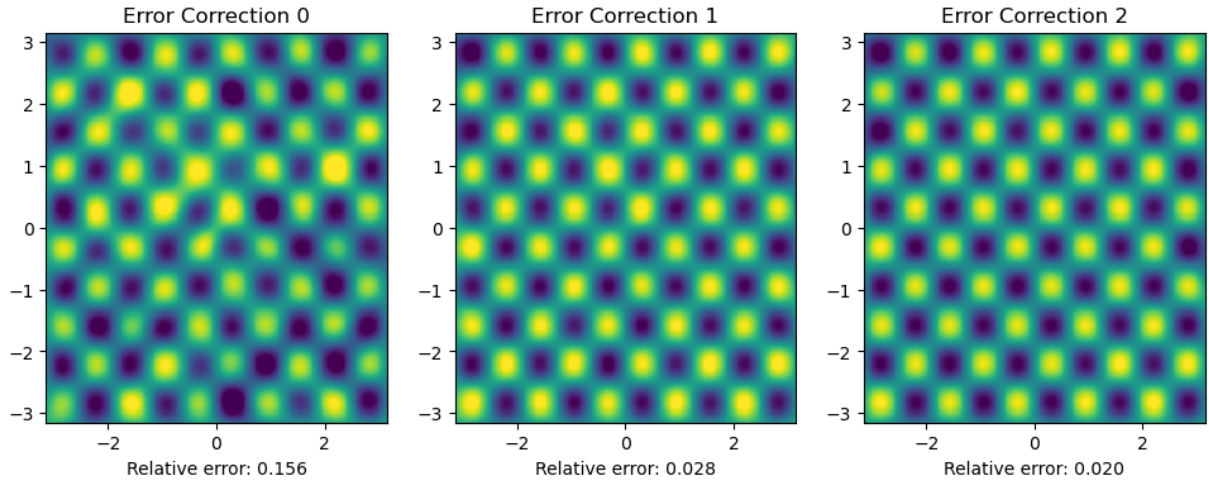}
\caption{Numerical solutions $\mathcal{N}^{(0)}, \mathcal{N}^{(1)}$ and $\mathcal{N}^{(2)}$ to (\ref{p1})} \label{N1}
\end{figure}

\vspace{0.5in}

\begin{figure}[h!]
\centering
\includegraphics[width=\textwidth]{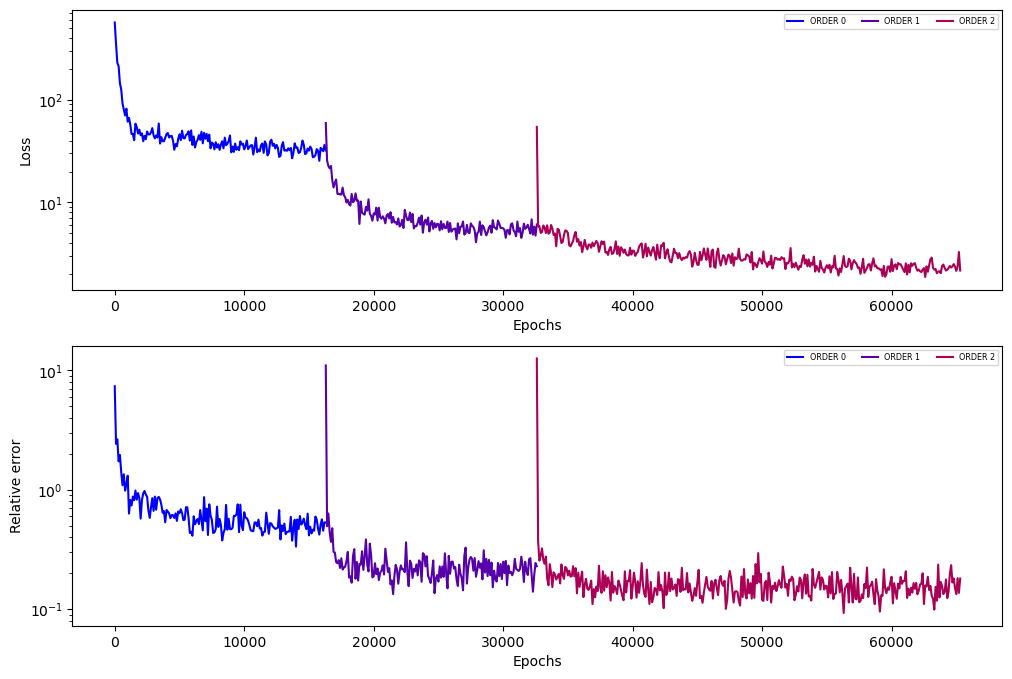}
\caption{Per-epoch loss and relative errors for numerical solutions to (\ref{p1})} \label{N1m}
\end{figure}

\newpage

\subsection{2D Poisson's Equation}

\subsubsection*{(i)}

\begin{figure}[h]
\centering
\includegraphics[width=0.5\textwidth, trim = {24.5cm 0 0 0}, clip]{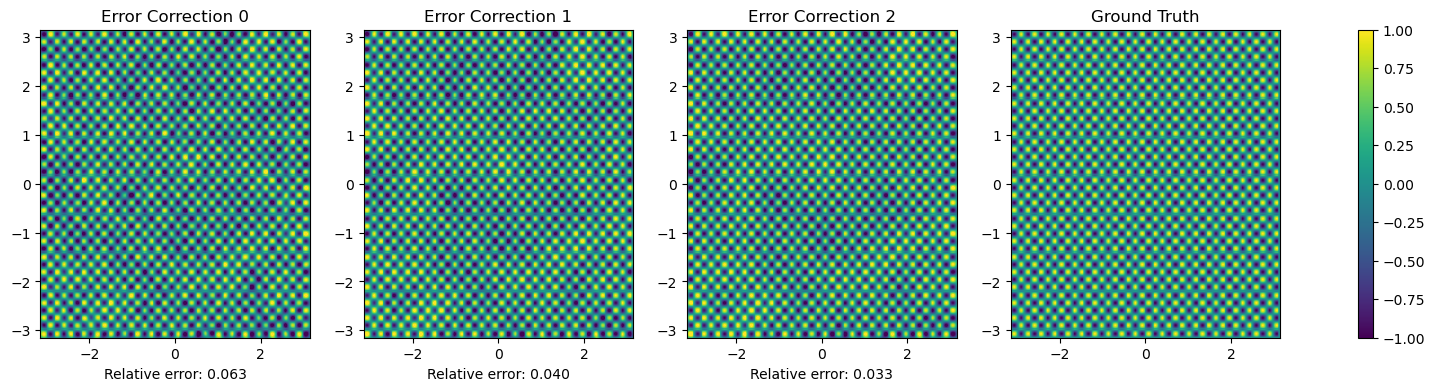}
\caption{$\phi(x,y)=\sin(20x)\sin(20y)$} \label{phi2}
\end{figure}

Figure \ref{phi2} shows the solution to Poisson's equation:

\begin{equation} \label{p2}
    \begin{cases}
        \nabla^2\phi & = -800\sin(5x)\sin(5y) \text{ in $\Omega=[-\pi,\pi]\times[-\pi,\pi]$} \\
        \phi & =  0\text{ on $\partial\Omega$}
    \end{cases}
\end{equation}

Due to the highly oscillatory nature of the solution, a neural network will struggle to accurately capture its structure. This is demonstrated in Figure \ref{toohard}, where the approximation cannot account for so many peaks and troughs in the solution.

\begin{figure}[h]
\centering
\includegraphics[width=0.4\textwidth]{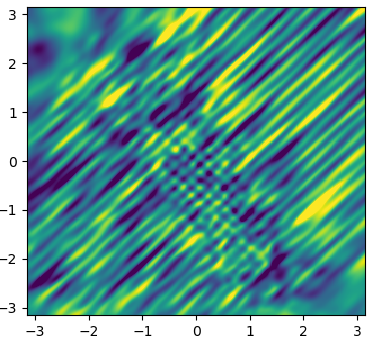}
\caption{Naive attempt at a numerical solution to (\ref{p2})} \label{toohard}
\end{figure}

To obtain a realistic solution, we apply a Gaussian random Fourier feature mapping to the input, before passing it through the network. After a simple sweep of values, we take $\Sigma = 1$ and $n=256$, as defined in (\ref{FFmap}). Figures \ref{N2} and \ref{N2m} show similar trends to those in the previous experiment. 

\newpage

\begin{figure}[h!]
\centering
\includegraphics[width=\textwidth, trim = {0 0 11.5cm 0}, clip]{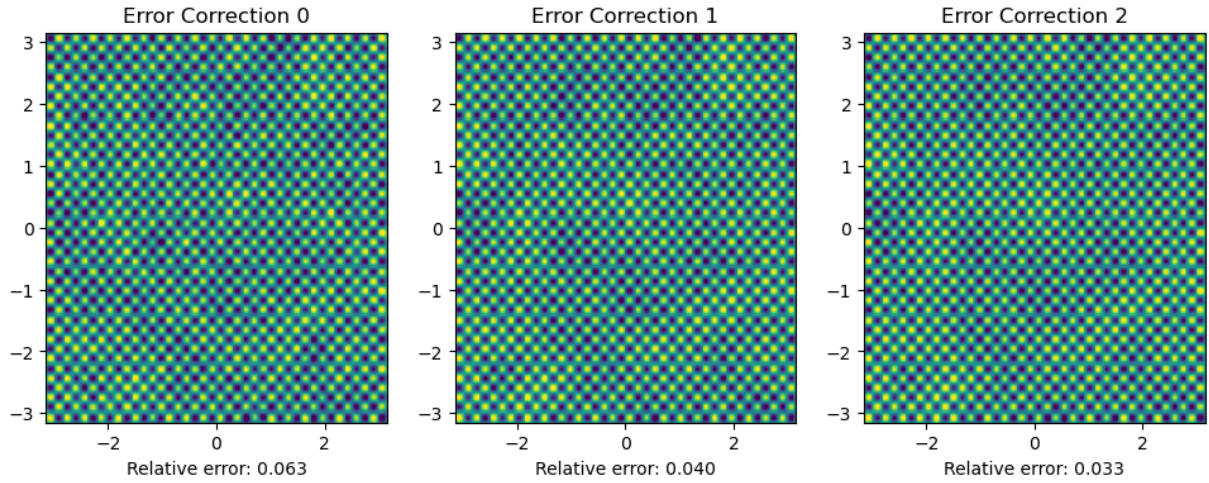}
\caption{Numerical solutions $\mathcal{N}^{(0)}, \mathcal{N}^{(1)}$ and $\mathcal{N}^{(2)}$ to (\ref{p2}), trained using random Fourier features with $\Sigma=1$ and $n=256$} \label{N2}
\end{figure}

\vspace{0.5in}

\begin{figure}[h!]
\centering
\includegraphics[width=\textwidth]{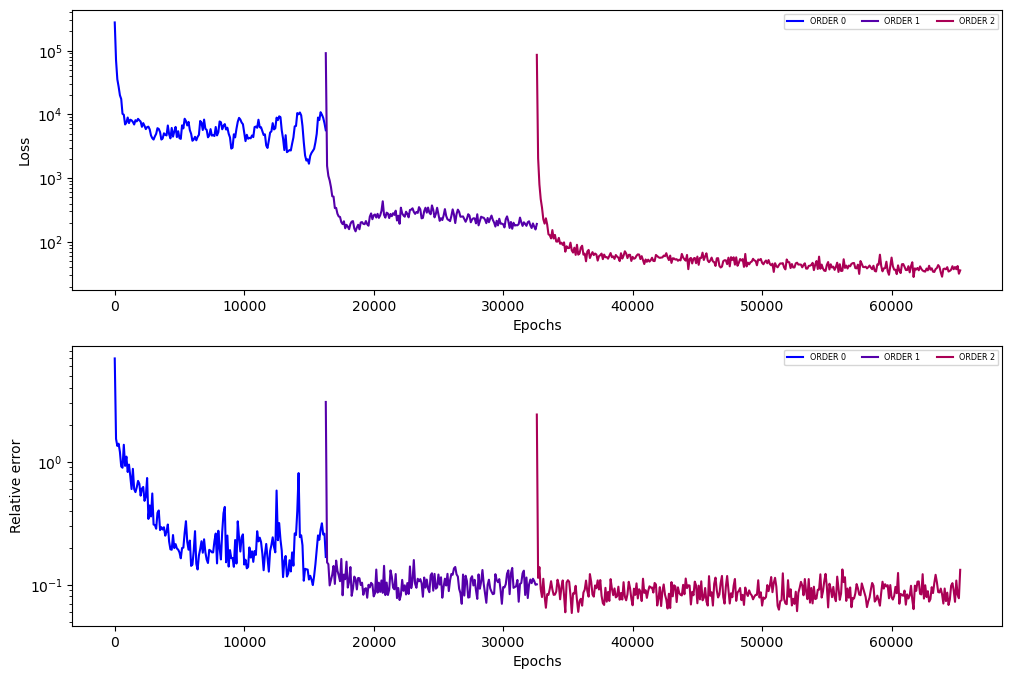}
\caption{Per-epoch loss and relative errors for numerical solutions to (\ref{p2})} \label{N2m}
\end{figure}

\newpage

\subsubsection*{(ii)}

\begin{figure}[h]
\centering
\includegraphics[width=0.5\textwidth]{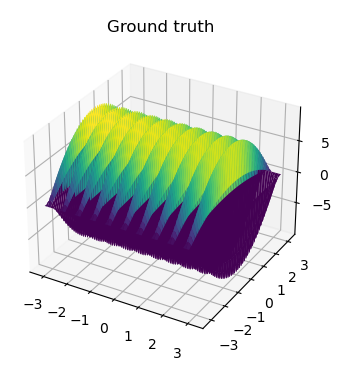}
\caption{$\phi(x,y)=(\pi^2-y^2)\sin(10x)$} \label{phi3}
\end{figure}

Figure \ref{phi3} shows the solution to Poisson's equation:

\begin{equation} \label{p3}
    \begin{cases}
        \nabla^2\phi & = (100y^2-100\pi^2-2)\sin(10x) \text{ in $\Omega=[-\pi,\pi]\times[-\pi,\pi]$} \\
        \phi & =  0\text{ on $\partial\Omega$}
    \end{cases}
\end{equation}

In Figure \ref{N3a}, we train a neural network $\mathcal{N}^{(0)}$ to approximate the solution to (\ref{p3}) for $2^{11}$ epochs, but we save its parameter states after $2^{10}$ epochs. These define a new network which we call $\mathcal{N}_0$. The fully-trained $\mathcal{N}^{(0)}$ achieves a reasonable relative error. Roughness is clearly visible in the plot.

In Figure \ref{N3b}, we plot the half-trained $\mathcal{N}_0$ on the left. As expected, it has not yet reached the accuracy of $\mathcal{N}^{(0)}$. However, we also initiate an error correction $\mathcal{N}_1$, of $\mathcal{N}_0$, that trains for another $2^{10}$ epochs. Thus, we produce an approximation $\mathcal{N}^{(1)}=\mathcal{N}_0+\mathcal{N}_1$ that has also trained for a total of $2^{11}$ epochs. This is significantly more accurate than $\mathcal{N}^{(0)}$, and the plot is visibly smoother. Figures \ref{N3a} and \ref{N3b} provide a clear exemplification of the immediate fruitfulness of a single error correction.

\newpage

\begin{figure}[h!]
\centering
\includegraphics[width=0.5\textwidth]{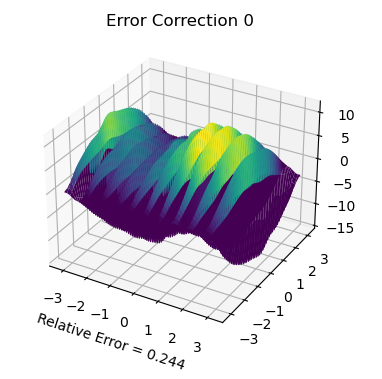}
\caption{Numerical solution $\mathcal{N}^{(0)}$ to (\ref{p3}), trained for $2^{11}$ epochs} \label{N3a}
\end{figure}

\vspace{0.5in}

\begin{figure}[h!]
\centering
\includegraphics[width=\textwidth]{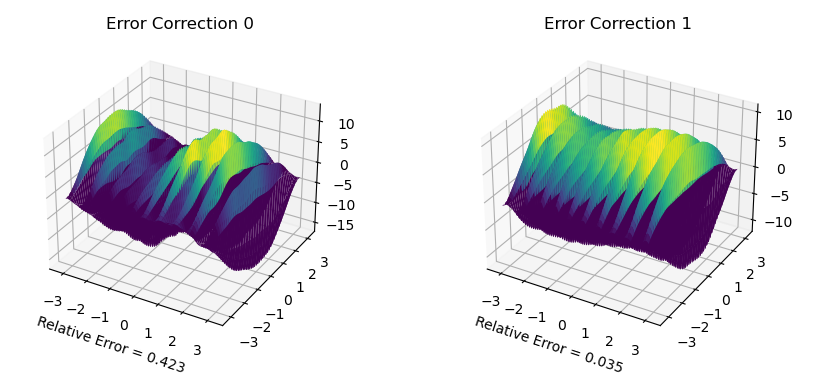}
\caption{Numerical solutions $\mathcal{N}_0$ and $\mathcal{N}^{(1)}$ to (\ref{p3}), trained for a total of $2^{11}$ epochs} \label{N3b}
\end{figure}

\newpage

\section{Discussion}

Our results do not endorse error correction as a tool to marginally reduce error across tens of corrections. Instead, they suggest training a network for half the allotted time, and devoting the other half to a single error correction. This can yield significantly more accurate results.

Error correction is not without cost however. In our implementation, we train correction networks on newly sampled points. This means that to obtain $\mathbf{F_k}[\mathcal{N}_k]$, we must first make $k$ forward passes of the new data through $\mathcal{N}_0,\mathcal{N}_1,...,\mathcal{N}_{k-1}$ and differentiate these to compute $\mathbf{F_{k-1}}[\mathcal{N}_{k-1}]$. The time complexity of producing a $k\textsuperscript{th}$ order approximation $\mathcal{N}^{(k)}$, assuming the number of epochs $E$ and batch size $B$ per correction, and optimisation costs, are kept constant across all corrections, is $\mathcal{O}\left(EB(k+1)^2\right)$. If we instead pass identical batches through each correction network, storing the $\mathbf{F_{k-1}}[\mathcal{N}_{k-1}]$ in memory, we can have a time complexity of $\mathcal{O}(EB(k+1))$, however the space complexity would be substantially increased.

\section{Further Work}

Over time, this study of neural network differential equation solvers naturally lent itself to hot-off-the-press topics in machine learning like sinusoidal representation networks \cite{SIREN} and random Fourier features \cite{FF}, for the simple reason that such concepts are inextricably linked through their applications. Outside of differential equations, neural networks as continuous parameterisations of discrete signals have immense potential in 3D shape representation, but also in image, video and audio representation and reconstruction. These problems may utilise neural networks as function approximators or, as we did, derivative approximators. There is no reason to suggest why the ideas of error correction cannot be employed here, and every reason to further explore the interplay of these techniques when applied to problems in computer vision.

\bibliographystyle{unsrt}
\bibliography{references}

\begin{thebibliography}{10}

\bibitem{NNDE_Pros}
I.E. Lagaris, A.~Likas, and D.I. Fotiadis.
\newblock Artificial neural networks for solving ordinary and partial
  differential equations.
\newblock {\em IEEE Transactions on Neural Networks}, 9(5):987--1000, 1998.

\bibitem{PINNS1}
Maziar Raissi, Paris Perdikaris, and George~Em Karniadakis.
\newblock Physics informed deep learning (part \uppercase{I}): Data-driven
  solutions of nonlinear partial differential equations.
\newblock {\em arXiv:1711.10561}, 2017.

\bibitem{PINNS2}
Maziar Raissi, Paris Perdikaris, and George~Em Karniadakis.
\newblock Physics informed deep learning (part \uppercase{II}): Data-driven
  discovery of nonlinear partial differential equations.
\newblock {\em arXiv:1711.10566}, 2017.

\bibitem{Sirignano_2018}
Justin Sirignano and Konstantinos Spiliopoulos.
\newblock {DGM}: A deep learning algorithm for solving partial differential
  equations.
\newblock {\em Journal of Computational Physics}, 375:1339--1364, 2018.

\bibitem{Quanta}
Anil Ananthaswamy.
\newblock Latest neural nets solve world's hardest equations faster than ever
  before.
\newblock {\em Quanta Magazine}, 2021.

\bibitem{Hornik1991}
Kurt Hornik.
\newblock Approximation capabilities of multilayer feedforward networks.
\newblock {\em Neural Networks}, 4(2):251--257, 1991.

\bibitem{Barron}
Andrew~R. Barron.
\newblock Universal approximation bounds for superpositions of a sigmoidal
  function.
\newblock {\em IEEE Transactions on Information Theory}, 39(3):930--945, 1993.

\bibitem{SIREN}
Vincent Sitzmann, Julien~N.P. Martel, Alexander~W. Bergman, David~B. Lindell,
  and Gordon Wetzstein.
\newblock Implicit neural representations with periodic activation functions.
\newblock In {\em Proceedings of the 34\textsuperscript{th} Conference on
  Neural Information Processing System}, 2020.

\bibitem{FF}
Matthew Tancik, Pratul~P. Srinivasan, Ben Mildenhall, Sara Fridovich-Keil,
  Nithin Raghavan, Utkarsh Singhal, Ravi Ramamoorthi, Jonathan~T. Barron, and
  Ren Ng.
\newblock Fourier features let networks learn high frequency functions in low
  dimensional domains.
\newblock In {\em Proceedings of the 34\textsuperscript{th} Conference on
  Neural Information Processing Systems}, 2020.

\bibitem{AKSH1}
Akshunna~S. Dogra.
\newblock Error estimation and correction from within neural network
  differential equation solvers.
\newblock {\em arXiv:2007.04433}, 2020.

\bibitem{AKSHetal}
Akshunna S.~Dogra \textit{et. al.}
\newblock Neural network differential equation solvers allow unsupervised error
  analysis and correction.
\newblock {\em under review}, 2023.

\bibitem{AKSH2}
Marios Mattheakis, David Sondak, Akshunna~S. Dogra, and Pavlos Protopapas.
\newblock Hamiltonian neural networks for solving equations of motion.
\newblock {\em Physical Review E}, 105(6), 2022.

\bibitem{AKSH3}
Akshunna~S. Dogra and William~T. Redman.
\newblock Local error quantification for neural network differential equation
  solvers.
\newblock {\em arXiv:2008.12190}, 2021.

\bibitem{AKSH4}
Akshunna~S. Dogra.
\newblock Dynamical systems and neural networks.
\newblock {\em arXiv:2004.11826}, 2020.

\bibitem{Hornik1989}
Kurt Hornik, Maxwell Stinchcombe, and Halbert White.
\newblock Multilayer feedforward networks are universal approximators.
\newblock {\em Neural Networks}, 2(5):359--366, 1989.

\bibitem{inequalities}
G.~H. Hardy, John~E. Littlewood, and George Pólya.
\newblock {\em Inequalities}.
\newblock Cambridge University Press, Cambridge, second edition, 1988.

\bibitem{GradDescent}
Simon~S. Du, Jason~D. Lee, Haochuan Li, Liwei Wang, and Xiyu Zhai.
\newblock Gradient descent finds global minima of deep neural networks.
\newblock In {\em Proceedings of the 36\textsuperscript{th} International
  Conference on Machine Learning}, 2019.

\bibitem{DasBlog}
Rudrajit Das.
\newblock {Recent Advances in Non-Convex Optimization for Deep Learning}.
\newblock \url{https://rudrajit15.github.io/posts/2018/09/blog-post-2/}.
\newblock Accessed: 2 May 2022.

\bibitem{ReLU}
Xavier Glorot, Antoine Bordes, and Yoshua Bengio.
\newblock Deep sparse rectifier neural networks.
\newblock In {\em Proceedings of the 14\textsuperscript{th} International
  Conference on Artificial Intelligence and Statistics}, 2011.

\bibitem{freqbias}
Ronen Basri, Meirav Galun, Amnon Geifman, David Jacobs, Yoni Kasten, and Shira
  Kritchman.
\newblock Frequency bias in neural networks for input of non-uniform density.
\newblock In {\em Proceedings of the 37\textsuperscript{th} International
  Conference on Machine Learning}, 2020.

\bibitem{specbias}
Nasim Rahaman, Aristide Baratin, Devansh Arpit, Felix Draxler, Min Lin, Fred~A.
  Hamprecht, Yoshua Bengio, and Aaron Courville.
\newblock On the spectral bias of neural networks.
\newblock In {\em Proceedings of the 36\textsuperscript{th} International
  Conference on Machine Learning}, 2019.

\bibitem{overparam}
Yuanzhi Li and Yingyu Liang.
\newblock Learning overparameterized neural networks via stochastic gradient
  descent on structured data.
\newblock In {\em Proceedings of the 32\textsuperscript{nd} Conference on
  Neural Information Processing Systems}, 2018.

\bibitem{freqbiasconvrate}
Ronen Basri, David Jacobs, Yoni Kasten, and Shira Kritchman.
\newblock The convergence rate of neural networks for learned functions of
  different frequencies.
\newblock In {\em Proceedings of the 33\textsuperscript{rd} Conference on
  Neural Information Processing Systems}, 2019.

\bibitem{NTK}
Arthur Jacot, Franck Gabriel, and Clement Hongler.
\newblock Neural tangent kernel: Convergence and generalization in neural
  networks.
\newblock In {\em Proceedings of the 32\textsuperscript{nd} Conference on
  Neural Information Processing Systems}, 2018.

\bibitem{AKSHwtr}
Akshunna~S. Dogra and William~T. Redman.
\newblock Optimizing neural networks via koopman operator theory.
\newblock {\em Advances in Neural Information Processing Systems \textbf{33}
  (NeurIPS 2020)}, 2020.

\bibitem{IVT}
Johannes Nicaise.
\newblock Lecture notes in real analysis, 2020.

\bibitem{ADAM}
Diederik~P. Kingma and Jimmy Ba.
\newblock Adam: A method for stochastic optimization.
\newblock In {\em Proceedings of the 3\textsuperscript{rd} International
  Conference for Learning Representations}, 2015.

\end{thebibliography}

\end{document}